\documentclass[12pt]{amsart}
\usepackage[utf8]{inputenc}
\usepackage[margin = 1 in]{geometry}

\usepackage{amsmath,comment}
\usepackage{amsthm}
\usepackage{amssymb}
\usepackage{mathrsfs}
\usepackage{mathtools,alltt,changepage}
\usepackage{tikz}
\usepackage{tikz-cd}
\usepackage{mathtools}
\usepackage{verbatim}
\usepackage{appendix}
\usepackage{xcolor}
\usepackage{longtable}
\usepackage{url,todonotes}
\usepackage{hyperref}
\usepackage{cleveref}
\hypersetup{colorlinks=true,citecolor=blue,urlcolor =black,linkbordercolor={1 0 0}}

\newtheorem{theorem}{Theorem}
\newtheorem{proposition}[theorem]{Proposition}
\newtheorem{lemma}[theorem]{Lemma}
\newtheorem{corollary}[theorem]{Corollary}

\theoremstyle{definition}
\newtheorem{defn}[theorem]{Definition}
\crefname{defn}{Definition}{Definitions}

\theoremstyle{remark}
\newtheorem*{remark}{Remark}

\newcommand{\on}{\operatorname}
\newcommand{\mb}{\mathbb}

\newcommand{\Spec}{\on{Spec}}

\newcommand{\wt}{\widetilde}

\newcommand{\BA}{\mathbb{A}}

\newcommand{\defeq}{\vcentcolon=}
\renewcommand{\cong}{\simeq}

\title{On the EKL-Degree of a Weyl Cover}
\author{Joseph Knight, Ashvin A.~Swaminathan, and Dennis Tseng}
\date{\today}

\AtEndDocument{\bigskip{\footnotesize%
  \textsc{Department of Mathematics, Purdue University, \mbox{West Lafayette, IN 47907}} \par
  \textit{E-mail address}, Joseph Knight: \texttt{knight63@purdue.edu} \par
  \addvspace{\medskipamount}
  \noindent\textsc{Department of Mathematics, Princeton University, \mbox{Princeton, NJ 08544}} \par
  \textit{E-mail address}, Ashvin A. Swaminathan: \texttt{ashvins@math.princeton.edu} \par
  \addvspace{\medskipamount}
  \noindent\textsc{Department of Mathematics, Harvard University, \mbox{Cambridge, MA 02138}} \par
  \textit{E-mail address}, Dennis Tseng: \texttt{dennisctseng@gmail.com}
}}

\subjclass[2010]{14M15, 55M25 (primary), and 14F42, 14G27 (secondary)}

\begin{document}

\maketitle

\vspace*{-0.3in}
\begin{abstract}
    More than four decades ago, Eisenbud, Khim\v{s}ia\v{s}vili, and Levine introduced an analogue in the algebro-geometric setting of the notion of local degree from differential topology. Their notion of degree, which we call the EKL-degree, can be thought of as a refinement of the usual notion of local degree in algebraic geometry that works over non-algebraically closed base fields, taking values in the Grothendieck-Witt ring.
%    The notion of EKL-degree provides an arithmetic refinement of the usual notion of degree in algebraic geometry.
    In this note, we compute the EKL-degree at the origin of certain finite covers $f\colon \mb{A}^n\to \mb{A}^n$ induced by quotients under actions of Weyl groups. We use knowledge of the cohomology ring of partial flag varieties as a key input in our proofs, and our computations give interesting explicit examples in the field of $\BA^1$-enumerative geometry.
    %In this note, we use well-known results on the cohomology of partial flag varieties to compute the EKL-degrees of various types of Weyl covers.
\end{abstract}

\section{Introduction} \label{sec-intro}

We work over a field $K$, which is arbitrary with characteristic not equal to 2 unless stated otherwise. Associated to a finite morphism $f\colon \BA^n\to \BA^n$ of $K$-varieties, we have the usual notion of its degree at the origin of the source, denoted by $\deg_0 f$. Refining this notion, Eisenbud, Khim\v{s}ia\v{s}vili, and Levine (see~ \cite{EL77,H77}) introduced a new notion of degree, namely the EKL-degree, which we denote by $\deg_0^{\on{EKL}} f$ and which is an element of the Grothendieck-Witt ring $\on{GW}(K)$.\footnote{We give precise definitions of $\on{GW}(K)$ and $\deg_0^{\on{EKL}}f$ in Section~\ref{sec-back}.} %The Grothendieck-Witt ring is generated by symmetric bilinear forms on $K$-vector spaces up to isomorphism, and the usual degree $\deg_0f$ can be recovered by taking the rank of the bilinear form $\deg_0^{\on{EKL}}f$.

If $K$ is algebraically closed, then the rank homomorphism defines an isomorphism of rings $\on{GW}(K) \overset{\sim}\longrightarrow \mb{Z}$, and $\deg_0^{\on{EKL}}f$ coincides with $\deg_0 f$. However, if $K=\mb{R}$, then the rank homomorphism $\on{GW}(\mb{R})\to \mb{Z}$ has kernel isomorphic to $\mathbb{Z}$, reflecting the fact that $\deg_0^{\on{EKL}}f$ also contains the data of the Brouwer degree of the underlying map of $\mb{R}$-manifolds. In general, $\deg_0^{\on{EKL}}f$ can be viewed as an enrichment of $\deg_0 f$ that contains interesting arithmetic data.

\begin{comment}
One of the main objectives of $\mathbb{A}^1$-enumerative geometry is to find enrichments over $K$ of results in classical enumerative geometry. The most fundamental example of such an enrichment is the local version of Morel's EKL-degree homomorphism due to Kass and Wickelgren, which assigns to each map $\pi \colon \BA^n \to \BA^n$ of $K$-varieties with an isolated zero at the origin an element denoted $\deg_0^{\on{EKL}} \pi$ of the Grothendieck-Witt ring $\on{GW}(K)$. %\footnote{For the definition of the EKL-degree, see~\cite[Section 0]{Mor12} and~\cite[Section 2]{KW19}. As explained in Section~\ref{sec-back}, we rely not on the definition but on an equivalent characterization to prove our results.}
As we explain at the end of Section~\ref{sec-back}, the EKL-degree generalizes the natural notions of degree that arise in topology and complex geometry.

In this note, we compute the EKL-degrees of various kinds of Weyl covers. The simplest example of a Weyl cover is given by the quotient of affine space $\BA^n$ by the action of the symmetric group $S_n$ on the coordinates. It is actually quite easy to compute the EKL-degree of this quotient map, and more generally, of quotients by finite groups that contain reflections. Indeed, we have the following result:
\end{comment}
%This local computation should agree with a suitable notion of (global) EKL-degree~(see Section \ref{sec-back}).

%In this note, we compute EKL-degrees of quotient maps induced by the action of a Weyl group on a root space.
In this paper, we compute EKL-degrees at the origin of maps $\BA^n \to \BA^n$ induced by actions of Weyl groups on root spaces. There are at least two motivations for performing this computation:
\begin{enumerate}
    \item The work of Eisenbud, Khim\v{s}ia\v{s}vili, and Levine from over four decades ago has experienced something of a revival in recent years through the field of $\BA^1$-enumerative geometry. By a result of Kass and Wickelgren \cite{KW19}, the algebraic notion of EKL-degree coincides with the local $\BA^1$-Brouwer-degree, which is a topological notion of degree that is of central importance to the field of $\BA^1$-enumerative geometry. Since the EKL-degree lends itself more readily to computation, their result provides a way of explicitly computing the local $\BA^1$-Brouwer-degree. In this regard, our computation constitutes an explicit example of a degree computation in $\BA^1$-enumerative geometry.
    \item In all of the maps $\BA^n \to \BA^n$ that we consider, the preimage of the origin is a finite scheme supported at the origin. Thus, it is natural to expect that the local degree that we compute in this paper should agree with a suitable notion of (global) $\BA^1$-Brouwer-degree, if such a notion were to be discovered. %(Note that the foundations of $\BA^1$-enumerative geometry are still being written; such a notion of global $\BA^1$-Brouwer-degree is expected to be given in the future preprint \cite{KLSW19preprint}.)
\end{enumerate}
%This local computation should agree with a suitable notion of (global) EKL-degree~(see Section \ref{sec-back}).

As a first example, consider the quotient map $\pi\colon \BA^n\to \BA^n/S_n\simeq \BA^n$ of affine space by the action of the symmetric group on the coordinates. The usual degree of $\pi$ is $\deg_0\pi = n!$, and it turns out that $\deg_0^{\on{EKL}} \pi=\frac{n!}{2} \cdot (\langle 1\rangle +\langle -1\rangle)$ for $n\geq 2$. This follows easily from the fact that $S_n$ contains a simple reflection, leading us to the following preliminary observation.

\begin{proposition}
\label{trivialthm}
Let $G$ be a finite group acting linearly on a finite-dimensional $K$-vector space $V$. If the ring $K[V]^G$ of $G$-invariants of $K[V]$ is a polynomial ring and $G$ contains a simple reflection, then the EKL-degree of  $\pi \colon \on{Spec} K[V] \to \on{Spec} K[V]^G$ is given by
\begin{align*}
\deg_0^{\on{EKL}} \pi = \frac{\deg_0\pi}{2} \cdot  (\langle 1\rangle +\langle -1\rangle).
\end{align*}
%\footnote{We explain the notation for elements of the Grothendieck-Witt ring in Section~\ref{sec-back}.}
\end{proposition}
%For example, Proposition~\ref{trivialthm} tells us that $\deg_0^{\on{EKL}} (\mathbb{A}_K^n\to \mathbb{A}_K^n/S_n) = \frac{n!}{2}\cdot (\langle 1\rangle+\langle -1\rangle)$.
For instance, \Cref{trivialthm} applies to quotients of root spaces by Weyl groups when $K$ is of characteristic zero by the Chevalley–Shephard–Todd theorem (see~\cite[(A)]{C55}) or in arbitrary characteristic when the Weyl group is of type $A$ or $C$ (see~\cite[Th\'{e}or\`{e}me]{D73}).

We can also compute EKL-degrees in situations where \Cref{trivialthm} does not apply. For example, we will show that the EKL-degree of the quotient map $\BA^4/(S_2 \times S_2) \to \BA^4/S_4$ is given by $4\cdot \langle 1\rangle + 2\cdot \langle -1\rangle$, so in particular, the EKL-degree is no longer a multiple of $\langle 1\rangle +\langle -1\rangle$. Generalizing this example, we prove the following:
\begin{theorem}
\label{partialquotient}
Let $n_1,\ldots,n_r$ be positive integers satisfying $n = \sum_{i = 1}^r n_i$. The EKL-degree of the map $\pi \colon \mathbb{A}_K^n\big/\prod_{i=1}^{r}S_{n_i}\to \mathbb{A}_K^{n}/S_n$
is given by
\begin{align*}
\deg_0^{\on{EKL}} \pi & = \frac{\deg_0\pi - a}{2} \cdot  (\langle 1\rangle +\langle -1\rangle) + a\cdot \langle 1 \rangle \\
& = \frac{1}{2}\left(\frac{n!}{\prod_{i=1}^{r}n_i!}+a\right)\cdot \langle 1\rangle + \frac{1}{2}\left(\frac{n!}{\prod_{i=1}^{r}n_i!}-a\right)\cdot \langle -1\rangle,
\end{align*}
where $a = \lfloor\frac{n}{2}\rfloor! \big/ \prod_{i=1}^{r}\lfloor\frac{n_i}{2}\rfloor!$ if at most one $n_i$ is odd and $a =
    0$ otherwise.
\end{theorem}
The proof of~\Cref{partialquotient} involves applying the Definition of the EKL-degree together with knowledge of the cohomology ring of partial flag varieties of type $A$. Motivated by this, we extend~\Cref{partialquotient} to apply to Weyl groups of other types as follows:
%Given that it is easy to compute the EKL-degrees of quotients like $\BA^n \to \BA^n/S_n$, one might ask whether it is possible to this for more complicated quotients, such as the subquotient $\BA^4/(S_2 \times S_2) \to \BA^4/S_4$, where left-hand factor of $S_2$ acts on the first two coordinates and the right-hand factor of $S_2$ acts on the last two coordinates. In the next theorem, we compute the EKL-degrees of such subquotients of Weyl covers.

\begin{theorem} \label{bigthm}
Let $K$ be a field of characteristic $0$. Let $G$ be a simple complex Lie group with root space $V/K$, and let $P \subset G$ be a parabolic subgroup. Let $W$ be the Weyl group of $G$, and let $W_P \subset W$ be the associated parabolic subgroup. Then the EKL-degree of the map $\pi \colon \Spec K[V]^{W_P} \to \Spec K[V]^W$ \mbox{is given by}
$$\deg_0^{\on{EKL}} \pi = \frac{\deg_0\pi - a_P}{2} \cdot (\langle 1 \rangle + \langle -1 \rangle) + a_P \cdot \langle \alpha \rangle,$$
where $\alpha \in K^\times$, and $a_P$ is equal to the number of cosets $\omega W_P \in W/W_P$ for which $\omega^{-1}\omega_0\omega \in W_P$, where $\omega_0 \in W$ is the longest word.
\end{theorem}

%We prove Theorem~\ref{bigthm} by translating the problem of computing a EKL-degree into a problem about cohomology ring of the partial flag variety $G/P$ (see Section~\ref{sec-proofs}).
The element $\alpha$ in the statement of Theorem~\ref{bigthm} depends on the choice of identifications of $\on{Spec}(K[V]^W)$ and $\on{Spec}(K[V]^{W_P})$ with $\mb{A}^{\dim(V)}$. Such identifications are equivalent to choosing generators of $\on{Spec}(K[V]^W)$ and $\on{Spec}(K[V]^{W_P})$ as polynomial rings over $K$. In particular, scaling a generator of $\on{Spec}(K[V]^W)$ by $\alpha'$ scales $\deg_0^{\on{EKL}} \pi$ by $(\alpha')^{-1}$, so there is always a choice of generators making $\alpha$ in \Cref{bigthm} equal to 1.

 In the type-$A$ case (i.e.,~\Cref{partialquotient}), we show that taking the obvious choice of generators using elementary symmetric functions yields $\alpha=1$.
On the other hand, the number $a_P$ in the statement of~\Cref{bigthm} can be computed explicitly in all cases, as we demonstrate in the following result:
\begin{proposition}\label{cor-0groups}
We have $a_P = 0$ in Theorem~\ref{bigthm} except in the following cases, tabulated according to the Dynkin diagrams of $W$ and the parabolic subgroup $W_P$:
\begin{center}
    \begin{tabular}{|c|c|c|}
    \hline
      $W$ & $W_P$ & $a_P$ \\ \hline
      $A_n$ & $\amalg_{i = 1}^r A_{n_i}$ with $n = \sum_{i = 1}^r n_i$ and $\#\{\text{odd }n_i\} \leq 1$ & $\lfloor\frac{n}{2}\rfloor! \big/ \prod_{i=1}^{r}\lfloor\frac{n_i}{2}\rfloor!$\\
      $D_{2n+1}$ & $D_{2n}$ & 2 \\
        $E_6$ & $D_5$ & 3 \\
        $E_6$ & $D_4$ & 6 \\
        \hline
    \end{tabular}
\end{center}
\end{proposition}
In particular, for all pairs $(G,P)$ not tabulated in \Cref{cor-0groups}, the EKL-degree of the map $\pi \colon \Spec K[V]^{W_P} \to \Spec K[V]^W$ \mbox{is given by}
$$\deg_0^{\on{EKL}} \pi = \frac{\deg_0\pi}{2} \cdot (\langle 1 \rangle + \langle -1 \rangle).$$

\begin{comment}
Finally, it is natural to ask what can be said about the number $\alpha$ that appears in the statement of Theorem~\ref{bigthm}. The following theorem says that we can take $\alpha = 1$ when $G = \mathrm{SL}_n$, in which case every parabolic subgroup can be expressed as $\prod_{i = 1}^r \on{SL}_{n_i}$, where $n = \sum_{i = 1}^r n_i$.
\begin{theorem}
\label{partialquotient}
Let $K$ be any field. Let $n_1,\ldots,n_r$ be positive integers satisfying $n = \sum_{i = 1}^r n_i$. Then the EKL-degree of the map $\pi_2 \colon \mathbb{A}_K^n\big/\prod_{i=1}^{r}S_{n_i}\to \mathbb{A}_K^{n}/S_n$
is given by
\begin{align*}
\deg_0^{\on{EKL}} \pi_2 & = \frac{\deg_0\pi_2 - a}{2} \cdot  (\langle 1\rangle +\langle -1\rangle) + a\cdot \langle 1 \rangle \\
& = \frac{1}{2}\left(\frac{n!}{\prod_{i=1}^{r}n_i!}+a\right)\cdot \langle 1\rangle + \frac{1}{2}\left(\frac{n!}{\prod_{i=1}^{r}n_i!}-a\right)\cdot \langle -1\rangle,
\end{align*}
where $a = \lfloor\frac{n}{2}\rfloor! \big/ \prod_{i=1}^{r}\lfloor\frac{n_i}{2}\rfloor!$ if at most one $n_i$ odd and $a =
    0$ otherwise.
\end{theorem}
\begin{remark}
    Note that Theorem~\ref{partialquotient} says that the result of Theorem~\ref{bigthm} holds for any field $K$ (not just fields of characteristic $0$).
\end{remark}
\end{comment}

\subsection*{Acknowledgments}

\noindent This work was supervised by Kirsten Wickelgren at the 2019 Arizona Winter School. The first author was supported by Purdue University, and the second and third authors were supported by NSF Graduate Research Fellowship Grants 1656466 and 1745303. We would like to thank Kirsten Wickelgren and Matthias Wendt for offering advice and guidance and for engaging in many enlightening discussions. We thank the anonymous referee for several useful comments and suggestions. We would also like to acknowledge Nicholas Kuhn, Victor Petrov, J.~D.~Quigley, Jason Starr, James Tao, and Libby Taylor. We used {\tt SageMath} for explicit computations.

\section{Background Material} \label{sec-back}

Before we prove our results, we provide a brief exposition on Grothendieck-Witt rings and on the EKL-degree in the case of finite maps between affine spaces.

\subsection{The Grothendieck-Witt Ring}
\label{GWdefsec}
We recall the definition of the Grothendieck-Witt ring of $K$ and introduce notation for specifying its elements. For more information about Grothendieck-Witt rings, see~\cite{EKM08,Lam05,WW19}.
\begin{defn}
\label{GWdef1}
Denoted by $\on{GW}(K)$, the Grothendieck-Witt ring of $K$ is defined to be the group completion of the semi-ring (under the operations of direct sum and tensor product) of isomorphism classes of symmetric nondegenerate bilinear forms on finite-dimensional vector spaces valued in $K$.
\end{defn}

%In addition to the abstract definition of $\on{GW}(K)$ given in \Cref{GWdef1}, we shall find it useful to have the following explicit presentation of $\on{GW}(K)$.

\begin{defn}
For $u \in K^\times$, define $\langle u \rangle \in \on{GW}(K)$ to be the class of the nondegenerate symmetric bilinear form that sends $(x,y) \in K^2$ to $u\cdot xy \in K$.
\end{defn}

The Grothendieck-Witt ring is generated by the classes $\langle u\rangle$ for $u\in K^{\times}$ and its relations can be written down explicitly \cite[Theorem 4.7]{EKM08}, giving a concrete presentation.

%\begin{theorem}
%[\protect{\cite[Theorem 4.7, p.~23]{EKM08}}]
\label{GWrelations}

The class of the bilinear form $\begin{pmatrix} 0 & 1 \\ 1 & 0\end{pmatrix}$ in $\on{GW}(K)$ is known as the hyperbolic form and can be expressed as $\langle 1 \rangle + \langle -1 \rangle$. It is easy to deduce from~\cite[Theorem 4.7, p.~23]{EKM08} that the product of the class of the hyperbolic form with any element in $\on{GW}(K)$ is an integral multiple of the class of the hyperbolic form.

\subsection{The EKL-Degree}
\label{A1def}

In this subsection, we recall the definition of the EKL-degree, where we will largely follow \cite[Section 1]{KW19}. Let $f\colon \BA^n\to \BA^n$ be a morphism sending $0$ to $0$ such that $0\in f^{-1}(0)$ is isolated in its fiber, and let $(f_1,\ldots,f_n)$ be its component functions. %Suppose that $f^{-1}(0)$ is supported at $0 \in \BA^n$.
\begin{defn}
\label{Qdef}
The local algebra of $f$ at $0 \in \mb{A}^n$ is $Q(f) \defeq K[x_1,\ldots,x_n]_{m_0}/(f_1,\ldots,f_n)$, where $m_0$ is the maximal ideal of the point $0$. The distinguished socle element is $E(f) \defeq \det(a_{ij}) \in Q(f)$ where $a_{ij}\in K[x_1,\ldots,x_n]$ are polynomials such that $f_i= \sum_j{a_{ij}\cdot (x_j-a_j)}$.
\end{defn}

\begin{defn}
\label{phidef}
To a linear functional $\phi\colon Q(f)\to K$, we can associate a symmetric bilinear form $\beta_{\phi}$ on $Q(f)$ defined by $\beta_{\phi}(a,b)=\phi(ab)$.
\end{defn}

\begin{defn} \label{thmdef-a1main}
The EKL-degree of $f$ at $0 \in \mb{A}^n$, denoted by $\deg_0^{\on{EKL}} f$, is given by the class of the bilinear form $\beta_\phi$ in $\on{GW}(K)$, where $\phi$ is any linear functional sending the distinguished socle element $E(f)$ to 1.
\end{defn}

\Cref{thmdef-a1main} does not depend on the choice of $\phi$ \cite[Lemma 6]{KW19}. In the proof of~\Cref{partialquotient} only, we will make use of the Jacobian element, which is defined as follows:
\begin{defn}
The Jacobian element is $J(f) \defeq \on{det}\left(\left.\frac{\partial f_i}{\partial x_j}\right|_0\right) \in Q(f)$.
\end{defn}
The Jacobian element and distinguished socle element are related to each other by the equation $J(f) = (\dim_K Q(f)) \cdot E(f)$~\cite[(4.7) Korollar]{SS75}, so $J(f)$ contains the same information as $E(f)$ if the characteristic of $K$ does not divide the dimension of $Q(f)$ as a $K$-vector space.%\todo{Isn't rank actually just dimension as a $K$-vector space here?}

\subsection{Lie theory notation} \label{sec-not}
We fix notation regarding complex Lie groups and their associated Weyl groups. In this paper, $G$ denotes a simple complex Lie group and $P\subset G$ denotes a parabolic subgroup of $G$. The Weyl group $W$ of $G$ is generated by a set $\Sigma \subset W$, consisting of simple reflections, which is in bijection with the nodes of the Dynkin diagram of $G$. If $B$ is a minimal parabolic subgroup of $G$, then there is a bijection between parabolic subgroups $P$ containing $B$ and subsets $\Sigma_P$ of $\Sigma$. The parabolic subgroup $W_P \subset W$ determined by such a $P$ is the subgroup of $W$ generated by the corresponding subset of $\Sigma$. A parabolic subgroup $P \subset G$ is \emph{proper} if $P \neq G$, or equivalently if $\Sigma_P \neq \Sigma$. A proper parabolic subgroup $W_{P}$ is said to be \emph{maximal} if it is not properly contained in any other proper parabolic subgroup.

 The \emph{length} $\ell(\omega)$ of an element $\omega \in W$ is defined to be the length of the shortest expression of $\omega$ as a product of the elements of the generating set $\Sigma$. The \emph{longest word}, denoted by $\omega_0$, is the unique element of maximal length in $W$.

\begin{remark}
In this paper, we restrict our consideration to the case of simple complex Lie groups. However, the classification of simple complex Lie groups is the same as the classification of split simple algebraic groups over fields of arbitrary characteristic (see~\cite[Theorem~23.25 and Theorem~23.55]{MR3729270}). Furthermore, the actions of the Weyl groups on the root spaces are the same for both simple complex Lie groups and split simple algebraic groups. Thus, our results can be extended to hold for the case of split simple algebraic groups, with the caveat that Chow groups, as opposed to singular cohomology groups, must be used in the proof.
\end{remark}

\section{Proofs of the Results} \label{sec-proofs}

For all of the maps $\pi\colon \mb{A}^n\to \mb{A}^n$ that we consider, $\pi^{-1}(0)$ is supported at the origin. This is because the orbit of $0\in \mb{A}^n$ under a linear group action is just the origin. %In particular, this means $\deg_0^{\on{EKL}}\pi$ can be evaluated using~\Cref{thmdef-a1main}.

\subsection{Proof of Proposition~\ref{trivialthm}} \label{sec-prooftriv}
Here, $K$ is an arbitrary field of characteristic not equal to 2. Since $G$ contains a simple reflection $r$, the map $\pi$ factors through $\on{Spec}(K[V]^r)$:
\begin{align*}
    \pi \colon \on{Spec}(K[V]) \to \on{Spec}(K[V]^r) \to \on{Spec}(K[V]^G).
\end{align*}
It is easy to check (for example using~\cite[Section~1]{KW19}) that $$\deg_0^{\on{EKL}}(\on{Spec}(K[V])\to \on{Spec}(K[V])^r) = \langle 1\rangle+\langle -1\rangle.$$ By the fact that EKL-degrees are multiplicative in compositions,\footnote{Given the connection between the EKL-degree and the local $\BA^1$-Brouwer-degree established in~\cite{KW19}, the multiplicativity of the EKL-degree in compositions can be deduced from the corresponding statement for the local $\BA^1$-Brouwer-degree in the $1$-dimensional case (see~\cite[Section 2]{KW20}). We suspect that this fact is known in arbitrary dimension to experts, but we have included a purely algebraic proof in \Cref{sec:composition} for the sake of clarity.} we have that $$\deg_0^{\on{EKL}} \pi = (\langle 1\rangle+\langle -1\rangle) \cdot \deg_0^{\on{EKL}}(\on{Spec}(K[V])^r \to \on{Spec}(K[V]^G)).$$
As mentioned in~\Cref{GWdefsec}, it follows from the presentation of the Grothendieck-Witt ring in~\cite[Theorem 4.7]{EKM08} that any product with the class of the hyperbolic form $\langle 1\rangle+\langle -1\rangle$ is actually an integral multiple of the class of the hyperbolic form. Thus, there is some integer $N$ such that $\deg_0^{\on{EKL}} \pi = N \cdot (\langle 1 \rangle + \langle -1 \rangle)$. Taking the rank of $\deg_0^{\on{EKL}} \pi$, we find that $2N  = \deg_0\pi$, which is the desired result. \hspace*{\fill}\qed
\vspace*{0.2cm}

It turns out to be more efficient from an expository standpoint to prove Theorem~\ref{bigthm} and Proposition~\ref{cor-0groups} before Theorem~\ref{partialquotient}, so we order the remaining proofs accordingly.

\subsection{Proof of Theorem~\ref{bigthm}}
Here, $K$ is an arbitrary field of characteristic $0$. We setup the proof in Section~\ref{sec:idea} and prove the result in Section~\ref{sec:proof}.
\subsubsection{The Idea of the Proof}
\label{sec:idea}
 Consider the algebra
\begin{align*}
    Q\defeq Q(\pi) \simeq K[V]^{W_P}/(K[V]^{W})^+,
\end{align*}
where for a graded ring $R$, we denote by $R^+$ its irrelevant ideal. Suppose that we can produce a $K$-linear functional $\phi\colon Q\to K$ sending the distinguished socle element $E \in Q$ to $1$. Then because $\pi^{-1}(0) = \{0\}$, it follows from Definition~\ref{thmdef-a1main} that $\deg_0^{\on{EKL}} \pi$ is the class in $\on{GW}(k)$ of the symmetric bilinear form $\beta_\phi\colon Q\times Q\to K$ defined by $\beta_\phi(a_1,a_2)\defeq\phi(a_1 a_2)$.

We now briefly sketch our idea for producing the desired functional $\phi$. The key observation is that $Q$ can be identified with the singular cohomology ring of a certain partial flag variety, so we can choose $\phi$ to be a certain scalar multiple of the integration map. Because the cohomology of this partial flag variety is has a basis given by classes of Schubert varieties, and because each Schubert variety has a dual Schubert variety, this forces $\beta_\phi$ to be a direct sum of copies of the bilinear forms $\begin{pmatrix} \alpha \end{pmatrix}$ and $\begin{pmatrix} 0 & 1\\ 1 & 0\end{pmatrix}$, where the multiplicity of each form depends on how many Schubert varieties are cohomologically equivalent to their dual Schubert variety.

By \cite[Proposition 29.2(a)]{B53} and the assumption in Theorem~\ref{bigthm} that $\operatorname{char}(K)=0$, the partial flag variety $F\defeq G/P$ has cohomology\footnote{Note that our partial flag variety is defined over $\mb{C}$, but we take its cohomology with coefficients in $K$.}
\begin{align*}
    H^{\bullet}(F,K)  = K[V]^{W_P}/(K[V]^{W})^+ = Q.
\end{align*}
\begin{comment}By \todo{cite}, the integral cohomology ring of $F$ is given by
\begin{align*}
    H^{\bullet}(F,\mathbb{Z}) = \mathbb{Z}[V]^{W_P}/(\mathbb{Z}[V]^W)^+, %\label{eq-IntCohombig}
    \end{align*}
so since $K$ is of characteristic $0$, we have that
$$H^{\bullet}(F,K) = H^{\bullet}(F,\mathbb{Z}) \otimes_{\mathbb{Z}} K = K[V]^{W_P}/(K[V]^{W})^+ = Q.$$
\end{comment}
%We claim tensoring the integration map with $K$ yields the functional. We need to check that the integration map is normalized according to $\phi$, specifically that the distinguished socle element is the class of a point in the cohomology $F$.
Integration over $G/P$ gives a functional $H^{\bullet}(F,K)\to K$. To understand the integration map explicitly, $H^{\bullet}(F,K)$ is graded by degree and is generated by the elements of degree 2. Integration maps the 1-dimensional $K$-vector space $H^{\dim_{\mathbb{R}}(F)}(F,K)$ isomorphically to $K$ and the other graded pieces (which are of smaller degree) of $H^{\bullet}(F,K)$ to zero. The class of a point $[\operatorname{pt}]\in H^{\dim_{\mathbb{R}}(F)}(F,K)$ integrates to 1.

We might try to take the functional $\phi$ to be the integration map on $H^{\bullet}(F,K)$, but to make this work, we would need to verify that the distinguished socle element $E$, viewed as an element of $H^{\bullet}(F,K)$, integrates to $1$. This also depends on the choice and ordering of polynomial generators of $K[V]^{W_P}$ and $K[V]^{W}$ providing isomorphisms $\on{Spec}(K[V]^{W_P})\simeq \on{Spec}(K[V]^{W})\simeq \BA^{\dim_K V}$. We verify this in the case where $G = \mathrm{SL}_n$ in Section~\ref{sec-sln}, where we used elementary symmetric functions as the generators in the invariant rings.

\subsubsection{The Proof}
\label{sec:proof}
Consider the integration map
\begin{align*}
    H^{\bullet}(F,K)\cong Q\to K.
\end{align*}
Let $\alpha \in K^\times$ be such that $\frac{1}{\alpha}$ is the integral of $E$, and let $\phi$ be such that $\frac{1}{\alpha} \cdot \phi$ is the integration map. We now compute the intersection pairing $\beta_\phi$ on $Q$. To do this, we use the following three facts (see~\cite[Section 2.1]{B05}):
\begin{enumerate}
\item the cohomology $H^{\bullet}(F,K)$ of $F$ has a basis given by the classes of the Schubert varieties;
\item Schubert varieties are indexed by cosets of $W/W_P$; and
\item the basis of Schubert varieties has a dual basis under the integration pairing, also given by Schubert varieties. The Schubert variety dual to the Schubert variety associated to the coset $\omega W_P$ is given by the coset $\omega_0\omega W_P$, where $\omega_0\in W$ is \mbox{the longest word.}
\end{enumerate}
It follows that the matrix of $\beta_\phi$ with respect to the basis of Schubert classes is block-diagonal, where the blocks are of two types: $\begin{pmatrix} \alpha \end{pmatrix}$ arising from self-dual Schubert classes, and $\begin{pmatrix} 0 & \alpha\\ \alpha & 0\end{pmatrix}$ arising from all other dual pairs. Note that the class of $\begin{pmatrix} \alpha \end{pmatrix}$ in $\on{GW}(K)$ is given by $ \langle \alpha \rangle$ and that~\cite[part (2) of Remark~1.15]{EKM08} implies that the class of $\begin{pmatrix} 0 & \alpha\\ \alpha & 0\end{pmatrix}$ is given by $\langle 1 \rangle + \langle -1 \rangle$.

Let $a_P$ be the number of self-dual Schubert classes. Then the number of other dual pairs of Schubert classes is simply given by $\frac{1}{2}(\dim_{K} Q - a_P) = \frac{1}{2}(\deg_0\pi - a_P).$
The theorem now follows upon observing that $a_P$ is equal to the number of cosets $\omega W_P$ such that $\omega_0\omega$ belongs to the same coset, which is equivalent to saying that $\omega^{-1} \omega_0 \omega \in W_P$. \hspace*{\fill} \qed

\subsection{Proof of \Cref{cor-0groups}}

Here, $K$ is an arbitrary field of characteristic $0$. Suppose that the Dynkin diagram of $G$ is not any one of $A_n$, $D_{n}$ for $n$ odd, or $E_6$. Then the longest word $\omega_0$ is in the center of $W$ (see~\cite[part (XI) of each of Planches I-IX, p.~250-275]{Bou}), and the support of $\omega_0$ is full (in the sense that one requires every generator of $W$ to express $\omega_0$). It follows that $\omega^{-1}\omega_0\omega = \omega_0$ is not contained in any parabolic subgroup of $W$, so we must have that $a_P = 0$. We treat the remaining cases separately as follows.

\subsubsection{The $A_n$ Case}

In this case, the Weyl group of $G$ is $W = S_n$, and any parabolic subgroup $W_P \subset W$ is of the form $W_P = \prod_{i = 1}^r S_{n_i}$, where $n = \sum_{i = 1}^r n_i$. The longest word $\omega_0\in S_n$ is the permutation that sends $i$ to $n-i$ for every $i$. Recall that the number $a_P$ of self-dual Schubert classes is equal to the number of cosets $\omega \cdot \prod_{i=1}^{r}S_{n_i}$ such that $\omega_0\omega$ belongs to the same coset, which is further equal to the number of elements in the set $P$ of partitions of $\{1,\ldots,n\}$ into blocks $B_1,\ldots,B_r$ (not necessarily contiguous) of sizes $n_1,\ldots,n_r$ such that swapping $n-i$ for each $i$ preserves those blocks. If all of the blocks are of even size, then $\# P$ is equal to the number of partitions of $\{1,\ldots,\frac{n}{2}\}$ into blocks of size $\frac{n_i}{2}$. If some block has odd size, then that block must contain $\frac{n+1}{2}$ (in particular, $n$ must be odd) and must therefore be the only block of odd size. Thus, if there is a single block of odd size that contains $\frac{n+1}{2}$, then $\# P$ is equal to the number of partitions of $\{1,\ldots,\frac{n-1}{2}\}$ into blocks of size $\lfloor\frac{n_i}{2}\rfloor$, and $\# P = 0$ otherwise. So $a_P = \#P = \lfloor\frac{n}{2}\rfloor! \big/ \prod_{i=1}^{r}\lfloor\frac{n_i}{2}\rfloor!$ if at most one $n_i$ is odd and $a_P = \#P =
    0$ otherwise, as desired.

\subsubsection{The $D_n$ case}
We take $n \geq 5$ to be odd (note that when $n = 3$, we have $D_3 = A_3$). In this case, the Weyl group of $G$ has presentation
$$W = \langle r_1, \dots, r_n : (r_ir_j)^{m_{ij}}=1 \rangle,$$
where the $m_{ij}$ are defined by
$$    m_{ij}=
    \begin{cases}
      1, & \text{if}\ i=j \\
      2, & \text{if}\ (i,j)=(1,2) \text{ or } |i-j|>1 \text{ for } (i,j) \neq (1,3),(3,1)\\
      3, & \text{if}\ |i-j|=1 \text{ for } i,j \geq 2, \text{ or if } (i,j)=(1,3),(3,1)
    \end{cases}$$
The generator $r_k$ of $W$ correspond to the node $k$ of the Dynkin diagram of $D_n$ labeled below.
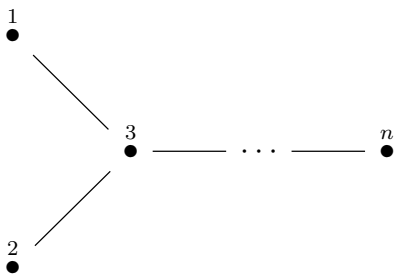
\begin{center}
\begin{tikzcd}
\stackrel{1}{\bullet} \ar[dr,dash] & & &   \\
 & \stackrel{3}{\bullet} \ar[r,dash] & \cdots \ar[r,dash]  & \stackrel{n}{\bullet}\\
 \stackrel{2}{\bullet} \ar[ur,dash] & & &
\end{tikzcd}
\end{center}
We now restate the following useful facts from Section~\ref{sec-not} in terms of the generators $r_k$:
\begin{itemize}
\item The length $\ell(\omega)$ of $\omega \in W$ is the length of the shortest expression of $\omega$ as a product of the generators $r_k$.
\item The longest word $\omega_0$ is an involution satisfying $\ell(\omega_0) = n^2-n$, and when $n$ is odd, $\omega_0$ acts by conjugation on the generators as follows: $\omega_0r_1\omega_0^{-1}=r_2$, and $\omega_0r_i\omega_0^{-1}=r_i$ for $i \geq 3$ (see~\cite[part (XI) of Planche IV, p.~257]{Bou}).
\item The proper parabolic subgroups of $W$ are precisely those subgroups of the form $W_{P_I} = \langle r_i : i \in I\rangle$, where $I \subsetneq \lbrace 1, \dots, n \rbrace$ is any subset.
\item The maximal parabolic subgroup $W_{P_I}$ is given by taking $\# I = n-1$.
\end{itemize}
The following lemma tells us that $a_{P_I} = 0$ unless $I = \{1, \dots, n-1\}$:
\begin{lemma}
\label{claim:Dn}
 If $\omega^{-1}\omega_0\omega \in W_{P_I}$ for some $\omega \in W$ and proper parabolic subgroup $W_{P_I}$, then $I= \lbrace 1, \dots, n-1 \rbrace$.
\end{lemma}
\begin{proof}
Note that for such an $\omega$ as in the statement of Lemma~\ref{claim:Dn}, any element $\omega' \in Z\omega$ also satisfies $(\omega')^{-1}\omega_0\omega' \in W_{P_I}$, where $Z$ is the centralizer of $\omega_0$ in $W$.

The following table shows how to reduce the length of a coset representative of $Z\omega$ by left-multiplication with elements of $Z$. In each row, the leftmost entry is a possible starting segment $b$ for $\omega$ expressed as a word $\omega = b \cdot c$, the middle entry is a re-expression of the starting segment $b$ that is more convenient for the purpose of length reduction, and the rightmost entry is a shortened segment $b' \in Z b$ with $\ell(b') < \ell(b)$. Before we demonstrate how to perform the length reduction, we narrow down the list of starting segments that we need to work with:
\begin{itemize}
\item First, because $\{r_i : i\geq 3\} \subset Z$, it is sufficient to consider starting segments $b$ that begin with $r_1$ or $r_2$.
\item Moreover, because $r_2r_1 \in Z$, and $r_2r_1r_2=r_1$, it is sufficient to consider starting segments $b$ that begin with $r_1$.
\item Finally, because $r_1r_2 \in Z$ and because $\{r_i : i \geq 4\}$ is contained in the centralizer of $r_1$, it is sufficient to consider starting segments $b$ that begin with $r_1r_3$.
\end{itemize}
Let $\sigma_{i,k}\in W$ be defined by
\begin{equation*}
    \sigma_{i,k}=
    \begin{cases}
      \prod_{j = i}^k r_j, & \text{if}\  i \leq k \\
      %r_i, & \text{if}\ i=k\\
      1, & \text{if}\ i>k
    \end{cases}
  \end{equation*}
  In terms of the $\sigma_{i,k}$, one verifies using the three itemized observations above that the possible starting segments that we need to work with are given by $r_1r_3r_1$, $r_1r_3r_2$, $r_1\sigma_{3,k}r_j$, $r_1\sigma_{3,k}r_j$, and $r_1r_3r_j$. We now present the table indicating how to reduce the lengths of words with these initial segments:
\begin{center}
\begin{tabular}{ |c|c|c|c| }
\hline
Starting Segment & Re-expression of Starting Segment & Shortened Segment & Conditions \\
 \hline
 $r_1r_3r_1$    & $r_3 \cdot r_1r_3$  & $r_1r_3$ & n/a \\
 \hline
 $r_1r_3r_2$  & $r_1r_3r_2 \cdot (r_2r_3)^3 = (r_1r_2r_3) \cdot r_2r_3 $ & $r_1r_3$ & n/a \\
 \hline
 $r_1\sigma_{3,k}r_j$   & $(r_1r_3r_j) \cdot \sigma_{4,k} $ & $r_1\sigma_{3,k}$ & $1 \leq j \leq 2$ \\
 \hline
 $r_1\sigma_{3,k}r_j $  & $ r_1r_3\sigma_{4,j+1}r_j\sigma_{j+2,k} = r_{j+1}\cdot r_1r_3\sigma_{4,k} $ & $r_1\sigma_{3,k} $ & $3 \leq j \leq k-1$ \\
 \hline
  $r_1r_3r_j$  & $r_j \cdot r_1r_3$ & $r_1r_3$ & $j>4$\\
 \hline
\end{tabular}
\end{center}\vspace*{0.2cm}
For example, the reduction in row 4 of the table is justified as follows: the defining relations of $W$ imply that $\sigma_{j+2,k}r_j=r_j \sigma_{j+2,k}$, and that $r_jr_{j+1}r_j=r_{j+1}r_jr_{j+1}$.

Since each row of the table constitutes a reduction in length, we have shown that if the conjugacy class of $\omega_0$ meets $W_{P_I}$, then there is an element
$$\omega \in S\defeq \lbrace r_1\sigma_{3,k} : 3 \leq k \leq n \rbrace$$
such that $\omega^{-1}\omega_0\omega \in W_{P_I}$. The length of the longest element of $S$ is $n-1$, so we deduce that
$$\ell(\omega^{-1}\omega_0\omega) \geq \ell(\omega_0)-2\cdot \ell(\omega) \geq (n^2-n)-2\cdot (n-1) = n^2-3n+2.$$
To finish the proof of the claim, it is enough to see that $\ell_k < n^2 - 3n + 2$, where $k<n$ and $\ell_k$ is the length of the longest element of the (unique) maximal parabolic subgroup not containing $r_k$. The maximal lengths in a Weyl group of type $A_r$ or $D_r$ are $\binom{r+1}{2}$ and $r^2-r$, respectively (see~\cite[part (I) of each of Planche I, p.~250 and Planche IV, p.~256]{Bou}). Using this fact together with the additivity of maximal lengths in products of Coxeter groups, we find that
  \begin{equation*}
   \ell_k=
    \begin{cases}
      \binom{n}{2}, & \text{if}\ 1 \leq k \leq 2 \\
      (k-1)^2-(k-1) + \binom{n-k+1}{2}, & \text{if}\ 2<k<n
    \end{cases}
  \end{equation*}
  It is easy to check in each case that $f(n,k) \defeq n^2-3n+2-\ell_k >0$ when $n \geq 5$. For example, in the case $2<k<n$, one readily checks that $f(n,k)$ is minimal when $k=n-1$, in which case $f(n,n-1) = 4n-10 > 0$. Thus we have proven the claim.
  \end{proof}

  Let $I = \lbrace 1, \dots ,n-1 \rbrace$, and let $P_I \subset G = \mathrm{SO}(2n)$ be an associated parabolic subgroup. We can realize the flag variety $G/P_I$ as a smooth quadric hypersurface $X$ of dimension $2n-2$. Indeed, under the obvious transitive action of $G$ on such a quadric, the subgroup $M=\mathrm{SO}(2n-2) \subset P_I \subset G$ (embedded in the standard way by acting as the identity on the last two coordinates) stabilizes a point $p \in X$. Since the stabilizer $M' \subset G$ of $p$ is parabolic and contains $M$, it follows by inspecting its Dynkin diagram that $M' = P_I$. By~\cite[Proof of Theorem 1.13]{Reid72}, we have that $$H^{n-1}(X,\mathbb{Z})=\mathbb{Z}L_1 \oplus \mathbb{Z}L_2,$$ where the $L_i$ are classes of linear subspaces on $X$ satisfying $L_1^2=L_2^2=1$ and $L_1 \cdot L_2 = 0$. Thus, there are exactly two self-dual classes, meaning that $a_{P_I} = 2$, as desired.

  \subsubsection{The $E_6$ Case} We will use {\tt SageMath} to compute this case. We label the nodes of the Dynkin diagram of $E_6$ as below:
  
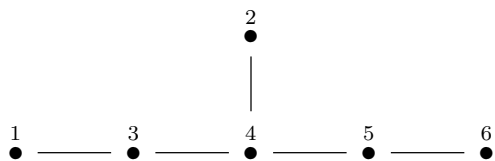
\begin{center}
    \begin{tikzcd}
     & & \stackrel{2}{\bullet} \ar[d,dash] & & \\
    \stackrel{1}{\bullet} \ar[r,dash] & \stackrel{3}{\bullet}\ar[r,dash] &\stackrel{4}{\bullet}\ar[r,dash] &\stackrel{5}{\bullet}\ar[r,dash] &\stackrel{6}{\bullet}
    \end{tikzcd}
\end{center}
Each parabolic subgroup $W_P$ of the Weyl group $W$ associated to $E_6$ corresponds to deleting a subset of the nodes in the Dynkin diagram. For each such $W_P$, we are interested in computing the the number $a_P$ of cosets $\omega W_P$ such that $\omega^{-1}\omega_0\omega\in W_P$, which is given as follows:
\begin{align}
\label{eq:E6number}
a_P = \frac{\#\{\omega \in W : \omega^{-1}\omega_0\omega \in W_P\}}{\# W_P}.
\end{align}
To compute the quantity~\eqref{eq:E6number}, we start with the following code, which initializes our Weyl group $W$ and longest word $\omega_0$ in {\tt SageMath}:
  \scriptsize
  \begin{adjustwidth}{0.5in}{0in}
  \begin{alltt}

  W=WeylGroup(["E", 6]);
  w0=E6.w0;
  \end{alltt}
\end{adjustwidth}
\normalsize
The first case to consider is where $W_P$ is a maximal parabolic subgroup; in terms of the Dynkin diagram, such a subgroup corresponds to deleting a single node from the Dynkin diagram of $E_6$. The number of distinct diagrams that arise from deleting a single node from the Dynkin diagram of $E_6$ is four, since deleting node 1 is equivalent to deleting node 6 and deleting node 3 is equivalent to deleting node 5 by symmetry. Deleting node 1 yields the Dynkin diagram
  
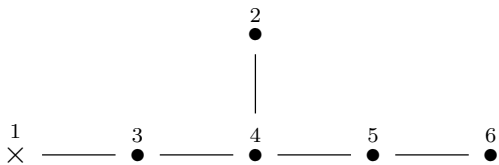
\begin{center}
    \begin{tikzcd}
     & & \stackrel{2}{\bullet} \ar[d,dash] & & \\
    \stackrel{1}{\times} \ar[r,dash] & \stackrel{3}{\bullet}\ar[r,dash] &\stackrel{4}{\bullet}\ar[r,dash] &\stackrel{5}{\bullet}\ar[r,dash] &\stackrel{6}{\bullet}
    \end{tikzcd}
\end{center}
which is the Dynkin diagram of $D_5$. For this choice of the maximal parabolic subgroup $W_P$, the following code computes $\#\{\omega \in W : \omega^{-1}\omega_0\omega \in W_P\}$ to be $5760$:
  \scriptsize
  \begin{adjustwidth}{0.3in}{0in}
  \begin{alltt}

    INPUT:
    i=0;
    for w in W:
        if len((w.inverse()*w0*w).coset_representative([2,3,4,5,6]).reduced_word())==0:
            i=i+1;
    print i;

    OUTPUT:
    5760
  \end{alltt}
\end{adjustwidth}
\normalsize
Moreover, by~\cite[part (X) of Planche IV, p.~257]{Bou}, we have that $\# W_P=\#W_{D_5}=2^4 \cdot 5! = 1920$ in this case, so using~\eqref{eq:E6number}, we find that
$$a_P = \frac{\#\{\omega \in W : \omega^{-1}\omega_0\omega \in W_P\}}{\# W_P} = \frac{5760}{1920} = 3.$$
In the remaining three cases where $W_P$ is a maximal parabolic arising from deleting nodes 2, 3, and 4, respectively, a similar block of code tells us that $\#\{\omega \in W : \omega^{-1}\omega_0\omega \in W_P\}=0$, implying that $a_P = 0$.

For the non-maximal parabolic subgroups $W_P$, we have that $a_P = 0$ whenever $W_P$ is contained in a maximal parabolic subgroup $W_{P'}$ for which $a_{P'} = 0$. Thus, it remains to consider those $W_P$ that are \emph{not} contained in any maximal parabolic subgroup $W_{P'}$ with $a_{P'} = 0$. Since there is only one maximal parabolic subgroup $P'$ with $a_{P'} \neq 0$, one verifies by inspection that the only such $P$ having the required property
%that the Dynkin diagram of $P/U(P)$ is $D_4$ and
is obtained by deleting the nodes labeled $1$ and $6$ from the Dynkin diagram of $E_6$, as illustrated below:
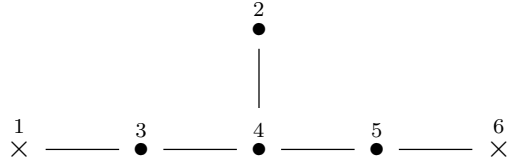
\begin{center}
    \begin{tikzcd}
     & & \stackrel{2}{\bullet} \ar[d,dash] & & \\
    \stackrel{1}{\times} \ar[r,dash] & \stackrel{3}{\bullet}\ar[r,dash] &\stackrel{4}{\bullet}\ar[r,dash] &\stackrel{5}{\bullet}\ar[r,dash] &\stackrel{6}{\times}
    \end{tikzcd}
\end{center}
For this choice of $W_P$, the following code computes $\#\{\omega \in W : \omega^{-1}\omega_0\omega \in W_P\}$ to be $1152$:
  \scriptsize
  \begin{adjustwidth}{0.5in}{0in}
\begin{alltt}

INPUT:
i=0;
for w in W:
    if len((w.inverse()*w0*w).coset_representative([2,3,4,5]).reduced_word())==0:
        i=i+1;
print i;

OUTPUT:
1152
\end{alltt}
\end{adjustwidth}
\normalsize
\vspace*{0.3cm}
 Moreover, by~\cite[part (X) of Planche IV, p.~257]{Bou}, we have that $\# W_P=\#W_{D_4}=2^3 \cdot 4! = 192$, so using~\eqref{eq:E6number}, we find that
 $$a_P = \frac{\#\{\omega \in W : \omega^{-1}\omega_0\omega \in W_P\}}{\# W_P} = \frac{1152}{192} = 6.$$
 This completes the proof of \Cref{cor-0groups}. \hspace*{\fill} \qed

\subsection{Proof of Theorem~\ref{partialquotient}} \label{sec-sln}
Here, $K$ is an arbitrary field of characteristic not equal to 2. The idea is to use the same strategy as in the proof of Theorem~\ref{bigthm}. For convenience, let $m_i=\sum_{j = 1}^i n_j$. Consider the partial flag variety $F\defeq F(m_1,\dots,m_r)$ parametrizing flags of $\mathbb{C}$-vector spaces $0\subset V_1\subset \cdots\subset V_r=\mathbb{C}^n$ where $V_i$ has dimension $m_i$. By~\cite[Proposition 31.1]{B53}, the integral cohomology ring of $F$ is given by
\begin{align}
    H^{\bullet}(F,\mathbb{Z}) = \frac{\mathbb{Z}[x_1,\ldots,x_n]^{\prod_{i=1}^{r}S_{n_i}}}{(\mathbb{Z}[x_1,\ldots,x_n]^{S_n})^+}. \label{eq-IntCohom}
    \end{align}
For any field $K$ (regardless of characteristic), we have that
$$Q:=Q(\pi) = \frac{K[x_1,\ldots,x_n]^{\prod_{i=1}^{r}S_{n_i}}}{(K[x_1,\ldots,x_n]^{S_n})^+} = H^{\bullet}(F,K).$$
%We claim tensoring the integration map with $K$ yields the functional. We need to check that the integration map is normalized according to $\phi$, specifically that the distinguished socle element is the class of a point in the cohomology $F$.
We want to take the functional $\phi$ to be the integration map on $H^{\bullet}(F,K)$, so we need to verify that the distinguished socle element $E:=E(\pi)$, viewed as an element of $H^{\bullet}(F,K)$, integrates to $1$. To do this, consider the element $\wt{E} \in \mathbb{Z}[x_1, \dots, x_n]$ defined by the formula for the distinguished socle element in \Cref{Qdef}. Viewing $\wt{E}$ as an element of $H^{\bullet}(F,\mathbb{Z})$ via the identification~\eqref{eq-IntCohom}, it is easy to see that the image of $\wt{E}$ under the map $H^{\bullet}(F,\mathbb{Z}) \to H^{\bullet}(F,K)$ is equal to $E$. It now suffices to show that $\wt{E}$ is equal to the class of a point in $H^{\bullet}(F,\mathbb{Z})$.

Notice that $\wt{E}\in H^{\text{top}}(F,\mathbb{Z})$ and that $H^{\text{top}}(F, \mathbb{Z}) \simeq \mathbb{Z}$. By~\cite[proof of Korollar 4.7]{SS75} (see also \cite[proof of Lemma 4]{KW19}), $E$ is nonzero independent of $K$, so we can vary $K = \mathbb{F}_p$ over all primes $p$ to see that the image of $\wt{E}$ in $H^{\text{top}}(F, \mathbb{F}_p)$ must be nonzero for each prime $p$. It follows that $\wt{E}$ is a generator of $H^{\text{top}}(F,\mathbb{Z})\simeq \mathbb{Z}$ and therefore agrees with the class of a point up to sign.

To determine the sign, it suffices to compute the sign of the Jacobian element $J:=J(\pi)$, taking $K=\mathbb{Q}$. We first consider the case where $n_i=1$ for every $i$. In this case, the Jacobian element is $J = \prod_{1\leq i<j\leq n}(x_i-x_j)$ by~\cite[Equation (1)]{LP02}; notice that $J$ is a Vandermonde determinant and can be expressed using the Leibniz formula as
\begin{equation} \label{eq-1}
J = \prod_{1\leq i<j\leq n}(x_{i}-x_{j}) = \sum_{\sigma \in S_n} \on{sign}(\sigma) \cdot \prod_{i = 1}^n x_{\sigma(i)}^{n-i}.
\end{equation}
On the other hand, the class of a point in $F$ is by definition given by $\prod_{i = 1}^n x_i^{n-i}$ (see~\cite[Section 1]{BJS93}). For any $\sigma \in S_n$, we have that
\begin{equation} \label{eq-2}
\sigma \cdot \prod_{i = 1}^n x_i^{n-i} = \prod_{i = 1}^n x_{\sigma(i)}^{n-i} = \on{sign}(\sigma) \cdot \prod_{i = 1}^n x_i^{n-i}.
\end{equation}
%The Jacobian element must generate $H^{\text{top}}(F, \mathbb{Q})$ as a $\mathbb{Q}$-vector space. Visibly, we see that $S_n$ acts on this vector space by the sign representation.
%It is easy to check the EKL-degree of $\mb{A}^n_K\to \mb{A}^n_K$ given by swapping two of the coordinates is $\langle -1\rangle$, which shows that $S_n$ acts on the vector space generated by $x_1^{n-1}x_2^{n-2}\cdots 1$ by the sign representation. Equivalently, one can use the fact that $S_n$ acts on the fundamental class of the full flag variety $H^{\text{top}}(F(1,\ldots,n))$ via the sign representation, as the representation on the entire cohomology ring is the regular representation \cite[Section 7.1]{CF13}.
It follows from combining~\eqref{eq-1} and~\eqref{eq-2} that $J = n! \cdot \prod_{i = 1}^n x_i^{n-i} \in H^{\text{top}}(F,\mathbb{Z})$\mbox{, so the signs agree.}% Note that it follows from the formula $J = \dim_K Q \cdot E$ that $\dim_K Q = n!$.

We next consider the general case where not every $n_i$ is equal to $1$. Consider the composition of maps
\begin{align} \label{eq-comp}
   \on{Spec}(\mb{Q}[x_1,\ldots,x_n])\to \on{Spec}(\mb{Q}[x_1,\ldots,x_n]^{\prod_{i=1}^{r}S_i})\to \on{Spec}(\mb{Q}[x_1,\ldots,x_n]^{S_n}).
\end{align}
The Jacobian element of the first map in~\eqref{eq-comp} is the product of the Jacobian elements of the maps $\on{Spec}(\mb{Q}[x_{m_{k-1}+1},\ldots,x_{m_k}]\to \on{Spec}(\mb{Q}[x_{m_{k-1}+1},\ldots,x_{m_k}]^{S_k})$ over $1 \leq k \leq r$. It then follows from the Chain Rule that the Jacobian element of $\on{Spec}(\mb{Q}[x_1,\ldots,x_n]^{\prod_{i=1}^{r}S_i})\to \on{Spec}(\mb{Q}[x_1,\ldots,x_n]^{S_n})$ is
\begin{align}
  J & = \prod_{1\leq i<j\leq n}(x_{i}-x_{j}) \bigg/\prod_{k = 1}^r \prod_{m_{k-1}+1\leq i<j\leq m_k}(x_{i}-x_{j}) \nonumber\\
  & =  \prod_{1\leq k<\ell\leq r}\prod_{i=1}^{n_k}\prod_{j=1}^{n_\ell}(x_{m_{k-1}+i}-x_{m_{\ell-1}+j}).\label{Jelement}
\end{align}
In words, this Jacobian element takes the same form as the product of differences $\prod_{1\leq i<j\leq n}(x_i-x_j)$, but instead of taking all pairwise differences $x_i-x_j$ for $i<j$, we instead take the pairs $i<j$ such that $i$ and $j$ are from different blocks, where we partition $\{1,\dots,n\}$ into contiguous blocks of size $n_1,\dots,n_r$. Now, we want to compare the Jacobian element with the class of a point in $F$. As before, we visibly see that swapping two variables from different blocks switches the sign and swapping two variables from the same block preserves the Jacobian element. Therefore, the same is true for the formula for the class of a point in $F$.

By~\cite[Section 2.1]{B05}, the class of a point in $F$ is the Schubert polynomial associated to the permutation \begin{equation} \label{eq-perm}
m_{r-1}+1,\ldots,m_{r},m_{r-2}+1,\ldots,m_{r-1},\cdots, 1,\ldots,m_{1}
\end{equation}
of the list $1, \dots, n$. In words, the permutation~\eqref{eq-perm} takes the numbers $1,\ldots,n$, splits them up into contiguous blocks of size $n_1,\ldots,n_r$ and reverses the order of the blocks (keeping the order within each block fixed). By~\cite[Block decomposition formula]{BJS93}, the Schubert polynomial associated to~\eqref{eq-perm} is given by $$\prod_{i=1}^{r}\left(\prod_{j=1}^{n_i}x_j\right)^{\sum_{k = i+1}^r n_k}$$
Expanding out~\eqref{Jelement} and keeping track of the signs, we find that
$$J = \frac{n!}{\prod_{i=1}^{r}n_i!} \cdot \prod_{i=1}^{r}\left(\prod_{j=1}^{n_i}x_j\right)^{\sum_{k = i+1}^r n_k}$$
so the signs agree.% Note that it follows from the formula $J = \dim_K Q \cdot E$ that $\dim_K Q = n!\big/\prod_{i = 1}^r n_i!$.

We deduce that $\alpha = 1$, so the theorem now follows from Theorem~\ref{bigthm} and \Cref{cor-0groups}.\hspace*{\fill}\qed

\appendix
\section{$\on{EKL}$-Degrees and Compositions}
\label{sec:composition}
In this section, we present a purely algebraic proof that the $\on{EKL}$-degree is multiplicative in compositions. As mentioned in Section~\ref{sec-prooftriv}, it is possible to prove this by utilizing the relationship between $\on{EKL}$-degrees and local $\BA^1$-Brouwer-degrees proven in~\cite{KW19}; nonetheless, we think the following proof is of independent interest. Here, $K$ is an arbitrary field of characteristic not equal to $2$.

The key building block in the proof is the following lemma.

\begin{lemma}
\label{rowoperation}
Let $f,g\colon \mathbb{A}^n\to \mathbb{A}^n$ be morphisms that send the origin to itself with $0\in f^{-1}(0)$ and $0\in g^{-1}(0)$ isolated in their fibers. Let $A\in  K^{n\times n}$ be a unipotent matrix, and let $L$ be the map defined by
\begin{align*}
    L\colon \mathbb{A}^n&\to \mathbb{A}^n\\
    x&\mapsto Ax,
\end{align*}
where $I_n$ is the $n\times n$ identity matrix. Then we have that
\begin{align*}
    \deg_0^{\on{EKL}}(f\circ g) = \deg_0^{\on{EKL}}(f\circ L \circ g).
\end{align*}
\end{lemma}

\begin{proof}
Let $t$ be a parameter, and consider the family of linear maps $L_t$ defined by the matrices $I_n+t(A-I_n)$. Since $A$ is unipotent, $\det(I_n+t(A-I_n))=1$, so $I_n+t(A-I_n)$ is invertible as a matrix with entries in the polynomial ring $K[t]$. Composing on the right and on the left with $g$ and $f$, respectively, yields the family of composite maps $f\circ L_t\circ g$. Upon applying the definition of $\on{EKL}$-degree (see~\Cref{Qdef,phidef,thmdef-a1main}) to each map in this family, we obtain a family of local algebras $Q(f \circ L_t \circ g)$, each equipped with a bilinear form $\beta_{\phi_t}$. Note that when $t=0$, the form $\beta_{\phi_0}$ represents the class $ \deg_0^{\on{EKL}}(f\circ g) \in \on{GW}(K)$, and when $t=1$, the form $\beta_{\phi_1}$ represents the class $\deg_0^{\on{EKL}}(f\circ L\circ g) \in \on{GW}(K)$. By a version of Harder's theorem (see~\cite[Lemma 30]{KW19}), the class in $\on{GW}(K)$ represented by $\beta_{\phi_t}$ is independent of $t\in K$, so $ \deg_0^{\on{EKL}}(f\circ g)=\deg_0^{\on{EKL}}(f \circ L_t \circ g)$ for each $t \in K$.
\end{proof}

\begin{remark}
Although \Cref{rowoperation} was stated with $A$ unipotent for maximal generality, we only need the case where $A$ is upper or lower triangular and unipotent in what follows.
\end{remark}

We are now ready to prove the desired result on the $\on{EKL}$-degree of a composition.

\begin{theorem}
\label{composition}
Let $f,g\colon \mathbb{A}^{n}\to \mathbb{A}^{n}$ be morphisms sending the origin to itself with $0\in f^{-1}(0)$ and $0\in g^{-1}(0)$ isolated in their fibers. Then, the $\on{EKL}$-degree of the composition $f \circ g$ is given by the product
\begin{align*}
    \deg_0^{\on{EKL}}(f\circ g) = (\deg_0^{\on{EKL}} f)\cdot ( \deg_0^{\on{EKL}} g)
\end{align*}
in $\on{GW}(K)$.
\end{theorem}

\begin{proof}
The idea is to use \Cref{rowoperation} to reduce to the case where $f$ and $g$ act on separate variables. To execute this idea, we first pad $f$ and $g$ with $n$ extra coordinates to obtain morphisms
\begin{align*}
    \widetilde{f},\widetilde{g}\colon \mathbb{A}^n \times \BA^n \to \mathbb{A}^n \times \BA^n
\end{align*}
where $\widetilde{f},\widetilde{g}$ send $(x,y)$ to $(f(x),y)$ and $(g(x),y)$ respectively. One readily verifies using Definition~\ref{thmdef-a1main} that the EKL-degree of a product of two morphisms $\BA^n \to \BA^n$ is the product of the EKL-degrees, so we have
\begin{align*}
    \deg_0^{\on{EKL}} (\widetilde{f}\circ \widetilde{g}) =  \deg_0^{\on{EKL}}(f\circ g).%, \qquad  \deg_0^{\on{EKL}} \widetilde{f}= \deg_0^{\on{EKL}} f, \qquad \deg_0^{\on{EKL}} \widetilde{g}= \deg_0^{\on{EKL}} g.
\end{align*}
Therefore, it suffices to show that
\begin{equation} \label{eq-whatwewant}
    \deg_0^{\on{EKL}} (\widetilde{f}\circ \widetilde{g}) =(\deg_0^{\on{EKL}} f)\cdot ( \deg_0^{\on{EKL}} g).
\end{equation}
To this end, repeated applications of \Cref{rowoperation} imply that the following four compositions of maps all have the same $\on{EKL}$-degree ($\deg_0^{\on{EKL}}$) at the origin. Each map takes in a point $(x,y)\in \mathbb{A}^{n}\times \mathbb{A}^{n}$ and outputs a point in $\mathbb{A}^{n}\times \mathbb{A}^{n}$, and each line contains a composition of three maps read from left to right.
\begin{align*}
& (x,y)\mapsto (g(x),y)\quad  (x,y) \mapsto (x,y) \quad(x,y)\mapsto (f(x),y) \\
   &(x,y)\mapsto (g(x),y)\quad  (x,y) \mapsto (x,x+y) \quad (x,y) \mapsto (f(x),y) \\
  &(x,y)\mapsto (g(x),y)\quad  (x,y) \mapsto (x-y,x) \quad (x,y) \mapsto (f(x),y) \\
&(x,y)\mapsto (g(x),y)\quad  (x,y) \mapsto (-y,x) \quad (x,y) \mapsto (f(x),y).
\end{align*}
The first composition above is given by $ \widetilde{f}\circ \widetilde{g}$, while the last composition sends $(x,y)$ to $(-f(y),g(x))$. To finish, we apply \Cref{rowoperation} to the following four compositions of maps:
\begin{align*}
& (x,y)\mapsto (-f(y),g(x)) \quad (x,y)\mapsto (x,y)\\
& (x,y)\mapsto (-f(y),g(x)) \quad (x,y) \mapsto (x+y,y)\\
& (x,y)\mapsto (-f(y),g(x)) \quad (x,y) \mapsto (y,-x+y)\\
& (x,y)\mapsto (-f(y),g(x)) \quad (x,y) \mapsto (y,-x).
\end{align*}
The first composition sends $(x,y)$ to $(-f(y),g(x))$, while the last composition sends $(x,y)$ to $(g(x),f(y))$. Therefore, we have shown that $ \widetilde{f}\circ \widetilde{g}$ has the same $\on{EKL}$-degree degree at the origin as the completely decoupled map that sends $(x,y)$ to $(g(x),f(y))$, which can be thought of as a product of two morphisms $\BA^n \to \BA^n$. It follows that~\eqref{eq-whatwewant} holds.
\end{proof}

\bibliographystyle{alpha} \bibliography{references}
\end{document}